\documentclass[12pt,reqno]{amsart}
\usepackage{amssymb}
\usepackage{txfonts}
\usepackage{amsmath}

 \newtheorem{thm}{Theorem}[section]
 \newtheorem{cor}[thm]{Corollary}
 \newtheorem{lem}[thm]{Lemma}

 \theoremstyle{definition}
 
 \theoremstyle{remark}
 \newtheorem{rem}[thm]{Remark}
 \numberwithin{equation}{section}

\begin{document}
	
	\title[]
	{On the $\kappa-$solutions of the Ricci flow on noncompact 3-manifolds  }
	
	\author{Liang Cheng, Anqiang Zhu}

	\dedicatory{}
	\date{}
	

	\thanks{}
	
	\address{Liang Cheng, School of Mathematics and Statistics, Central China Normal University,
		Wuhan, 430079, P.R. CHINA}
	
	\email{chengliang@mail.ccnu.edu.cn}
	
	\address{Anqiang Zhu, School of Mathematics and Statistics, Wuhan University,
		Wuhan, 430072, P.R. CHINA}
	
	\email{aqzhu.math@whu.edu.cn} \maketitle
	
	\begin{abstract}
	In this paper we prove that there is no  $\kappa$-solution of Ricci flow on 3-dimensional noncompact manifold with strictly positive sectional curvature and blow up at some finite time $T$ satisfying 	
$\int^T_0 \sqrt{T-t} R(p_0,t)dt< \infty$	for some point $p_0$. This partially confirms a conjecture of Perelman.
	\end{abstract}

	\section{Introduction}
	In \cite{P1} Perelman conjectured that
  there is no three-dimensional noncompact $\kappa$-solution with positive sectional
curvature and blow-up at finite time. Recall the $\kappa$-solution is a complete non-flat ancient solution of Ricci flow that is $\kappa$-noncollapsed on all scales
	and has bounded nonnegative curvature at each time slice.

	Our result partially confirms Perelman's conjecture with an extra condition. The main result of this paper is following theorem.
	
	\begin{thm}\label{main_1}
	There is no  $\kappa$-solution of Ricci flow on 3-dimensional noncompact manifold with strictly positive sectional curvature and blow up at some finite time $T>0$ satisfying 	
	\begin{align}\label{curvature_rate}
	\int^T_0 \sqrt{T-t} R(p_0,t)dt<\infty
	\end{align}
	for some point $p_0$.
	\end{thm}
	
\begin{rem}
	We remark that (\ref{curvature_rate}) holds if
there exists some point $p_0$ such that $ R(p_0,t)\leq \frac{C}{(T-t)^\alpha}$ with any $\alpha<\frac{3}{2}$.
\end{rem}	

Since  the only orientable 2-dimensional $k$-solution for Ricci flow is $S^2$, by Hamilton's strong maximum principle
Theorem \ref{main_1} implies the following  classification  theorem for the $\kappa$-solutions of 3-dimensional noncompact manifold satisfying the condition (\ref{curvature_rate}): 
\begin{cor}\label{main_2}
	Let $(M^3,g(t))$ be the $\kappa$-solution to Ricci flow blow-up at finite time $T$ on 3-dimensional noncompact manifold. If
	there exists a point $p_0\in M$ such that
	the condition (\ref{curvature_rate}) holds,
	$(M^3,g(t))$ must be  $S^2\times \mathbb{R}$ or its quotient.
\end{cor}
The main tool for the proof of Theorem \ref{main_1} is the singular
reduced entropy, which is obtained by taking the limit of the reduced length as the based point $(p_0,s_i)$ to the singular time. The singular reduced entropy was used by Naber \cite{N}
to show the asymptotic limit of Type I singularity of Ricci flow is shrinking soliton. By virtue of Hamiton's harnack inequality and Perelman's blow-down arguments, we can show that under the conditions of Theorem \ref{main_1} the singular reduced volume share the same value as time goes to blow-up time $T$ and goes to $-\infty$. Then the singular reduced volume is constant indepedent of time and hence the Ricci flow is a gradient Ricci soliton which contradicts to the postive sectional curvature. 

Finally, it is worth mentioning some related results. 
Hamilton \cite{RH2} showed that the only orientable 2-dimensional $k$-solution  is $S^2$.
Series works by Hamiton \cite{RH2},
 Daskalopoulos, Hamilton,
Sesum \cite{DHS}, Daskalopoulos and Sesum \cite{DH} classified 
the 2-dimensional ancient solutions with bounded curvature which is either
	 $S^2$,  $\mathbb{R}^2$, cigar steady soliton, 
	the King-Rosenau solution or their quotient.
	Ding \cite{Ding} proved that the only simply
	connected noncompact 3-dimensional $\kappa$-solution that forms a forward singularity of Type I is the
$S^2\times \mathbb{R}$. Zhang \cite{Z} and Hallgren \cite{Ha} independently showed 
only simply
connected noncompact 3-dimensional backward Type I $\kappa$-solution is the
$S^2\times \mathbb{R}$.		
	 Ni \cite{Ni} has proved that a closed Type I $\kappa$-solution with
	positive curvature operator of every dimension is a shrinking sphere or one of its
	quotients.	For the mean curvatue flow,
	 Brendle and Choi \cite{BC} recently proved that every noncompact ancient solution of mean curvature flow in $\mathbb{R}^3$ which is strictly convex and noncollapsed must be bowl soliton. 
	
	\section{preliminaries}
	
	 Perelman introduced in \cite{P1} the reduced entropy (i.e. reduced distance and reduced volume), which becomes one of powerful tools for studying Ricci flow.
The reduced entropy enjoys very nice analytic and geometric properties,
including in particular the monotonicity of the reduced volume. These
properties can be used, as demonstrated by Perelman, to show the limit of the suitable rescaled Ricci flows is a gradient shrinking soliton.

We recall some basic formulas and
	properties about reduced entropy in \cite{P1}.
	Let $g(t)$ solves the Ricci flow
	\begin{align}\label{Ricci_flow}
	\frac{\partial g}{\partial t}=-2Rc.
	\end{align}
	on $M\times
	(-\infty,T)$. Denote $h(\tau)=g(t)$ with $\tau(t)=T-t$.
The $l$-length is defined
	\begin{align}\label{reduced_l} l_{p,s}(q,\tau)=\inf\limits_{\gamma}\{\frac{1}{2\sqrt{\tau-s}}\int^{\tau}_s
\sqrt{\eta-s}(R_{h(\eta)}(\gamma(\eta))+|\gamma'(\eta)|_{h(\eta)}^2)d\eta\},
	\end{align}
where the infimum is taken over all curves $\gamma:[s,\tau]\to M$ with $\gamma(s)=p$
and $\gamma(\tau)=q$.
The reduced volume is definded as
\begin{align}\label{WFRV}
\mathcal {V}_{p,s}(\tau)=\int_{M} (4\pi (\tau-s))^{-\frac{n}{2}}
e^{-l_{p,s}(x,\tau)}
dvol_{h(\tau)}(x),
\end{align}

	\begin{thm}[Perelman]\label{FIN}
		Let $ l_{p,s}(q,\tau)$ be the reduced length defined in (\ref{reduced_l}) and $L_{p,s}(q,\tau)=2\sqrt{\tau-s}l_{p,s}(q,\tau)$. Then $ l_{p,s}(q,\tau)$ satisfies the following properties:
		\begin{align}
		&2\frac{\partial l_{p,s}}{\partial \tau}+|\nabla l_{p,s}|^2-R+\frac{l_{p,s}}{\tau-s}=0,\label{eq_l_1}\\
		&|\nabla l_{p,s}|^2=\frac{l_{p,s}}{\tau-s}-R-\frac{K}{(\tau-s)^{
				\frac{3}{2}}},\label{eq_l_2}\\
	&	(\frac{\partial }{\partial \tau}+\Delta)L_{p,s}\leq 2n,\label{eq_l_3}\\
		&	\frac{\partial l_{p,s}}{\partial \tau}-\Delta l_{p,s}+	|\nabla l_{p,s}|^2-R+\frac{l_{p,s}-n}{2(\tau-s)} \geq 0,\label{eq_l_4}\\
			&	2\Delta l_{p,s}-|\nabla l_{p,s}|^2+R+\frac{l_{p,s}-n}{\tau-s} \le 0.\label{eq_l_5}
		\end{align}
 For any $s$ and $\tau>s$, there exists $q^s_{\tau}\in M$ such that
$l_{p,s}(q^s_{\tau},\tau)\leq \frac{n}{2}$.	
The  reduced volume
defined in (\ref{WFRV}) is monotone non-increasing along
the backward Ricci flow $(M,h(\tau))$. 
Moreover, if the Ricci flow has nonnegative curvature operator, 
	\begin{align}
&|\nabla l_{p,s}|^2+R\leq c\frac{l_{p,s}}{\tau-s}\label{estimate3}\\
&l_{p,s}(x,\tau)\geq -l_{p,s}(y,\tau)-1+c\frac{d^2_{h(\tau)}(x,y)}{\tau-s}\label{estimate1}\\
&l_{p,s}^{\frac{1}{2}}(x,\tau) \leq l_{p,s}^{\frac{1}{2}} (y,\tau)+c\frac{d_{h(\tau)}(x,y)}{\sqrt{\tau-s}}\label{estimate2}
\end{align}
	\end{thm}
\begin{proof}
One can find the details of the proofs of Theorem \ref{FIN} for the version $s=0$	in \cite{P1} (also see in \cite{MT} and \cite{CCGGI}). Just shifting the time as $\tilde{h}(\tau-s)=h(\tau)$, one can get Theorem \ref{FIN}.
\end{proof}		

The following theorem by Perelman can be found the proof  in \cite{P1} (also see in \cite{MT} and \cite{CCGGI}).
	\begin{thm}[Perelman]\label{Perelman}
	If there exists
		$l\in C^{0,1}(M\times (0,+\infty))\cap W^{1,2}_{loc}(M\times (0,+\infty))$ with for $s=0$ (\ref{eq_l_1}), (\ref{eq_l_2}),(\ref{estimate3}) hold almost everywhere, (\ref{estimate1}),(\ref{estimate2}) hold and (\ref{eq_l_3}), (\ref{eq_l_4}) hold
		for distribution sense, morevoer $\mathcal {V}(\tau)\overset{def}{=}\int_{M} (4\pi \tau)^{-\frac{n}{2}}
		e^{-l(x,\tau)}
		dvol_{h(\tau)}(x)$ is constant under the backward Ricci flow $(M,h(\tau))$,
		then 
		 $(M,h(\tau))$ is gradient shriking soliton and
		$l$ is the smooth soliton function satisfying
		\begin{align}\label{3_7}
		Rc_{h}+\text{Hess} l=\frac{h}{2\tau},
		\end{align}
		and
		\begin{align}\label{3_6}
		2\Delta l- |\nabla l|^2 + R_{h}+ \frac{l-n}{\tau}
		= 0.
		\end{align}
	\end{thm}

	\section{The proof of Theoren \ref{main_1}}
	
Before presenting the proof of Theorem \ref{main_1}, we need the following $\epsilon$- regularity theorem for $\kappa$-solutions of Ricci flow.
\begin{lem}\label{regularity}
Let $(M,g(t))$ be the $\kappa$-solution to the Ricci flow on the 3-dimensional manifold for $t\in (-\infty,T]$. Denote $h(\tau)=g(t)$ with $\tau(t)=T-t$ and $\mathcal {V}_{p,s}$ be the reduced volume defined in (\ref{WFRV}). Then
there exists an universal constant $\epsilon_0$	
such that if $\mathcal {V}_{p,s}(\tau_0)\geq 1-\epsilon_0$, we have
 $r_{Rm}(p,s)\geq \epsilon_0 (\tau_0-s)$
where $r_{Rm}(p,s)=\sup\{r>0:\sup\limits_{B_{g(s)}(x,r)\times (s-r^2,s]}|Rm|\leq r^{-2}\}$.
\end{lem}
\begin{proof}
By shifting time and rescaling, we may assume $s=0$ and $\tau_0=1$.
We argue by contradiction. In this case we have that there exists a sequence of $\kappa$-solutions $(M_i,g_i(t),p_i)$  such that
\begin{align}\label{3_2}
\mathcal {V}_{p_i,0}(1)\geq 1-\frac{1}{i}
\end{align}
and
$$
r_{Rm}(p_i,0) \leq  \frac{1}{i}.
$$
Define $r_i=r_{Rm}(p_i,0)$. Consider the rescaled flows $(M_i,\tilde{g}_i(t),p_i)$
with $\tilde{g}_i(t)=r_i^{-2}g_i(r_i^2t)$. Clearly, $\tilde{r}_{Rm}(p_i,0)=1$.
 By Perelman's result, we have $R_{\tilde{g}_i}(q,0)\leq C(r)$ if
$d_{\tilde{g}_i(0)}(q,p_0)\leq r$. Then by Hamilon's harnack inequality, we get $R_{\tilde{g}_i}(q,t)\leq C(r)$
if
$d_{\tilde{g}_i(0)}(q,p_0)\leq r$.	This allows us to pass to a subsequence to derive a pointed limit
$$(M_i,\tilde{g}_i(t),p_i) \overset{C^{\infty}}{\to} (M_{\infty},\tilde{g}_{\infty}(t),p_{\infty}),$$
with
\begin{align}\label{radius}
\tilde{r}_{Rm}(p_\infty,0)=1.
\end{align}
Take the constant path $\gamma$ in $p_i$, we have $\tilde{l}_{p_i,0}(p_i,\tau)\leq \tilde{l}(\gamma)\leq \frac{3}{4}$ for $0\leq\tau\leq 1$,
where $\tilde{l}_{p_i,0}$ is the reduced length with respect to $\tilde{g}_i$. Then by  (\ref{estimate1}) and (\ref{estimate2})
\begin{align}\label{3_1}
-\frac{7}{4}+c\frac{d^2_{\tilde{h}_i(\tau)}(q,p_i)}{\tau}\leq \tilde{l}_{p_i,0}(q,\tau) \leq (\frac{\sqrt{3}}{2}+c\frac{d_{\tilde{h}_i(\tau)}(q,p_i)}{\sqrt{\tau}})^2.
\end{align}
Combining (\ref{estimate3}) and (\ref{3_1}), we obtain that passing by a
subsequence $\tilde{l}_{p_i,0} \overset{}{\longrightarrow} \tilde{l}_{p_{\infty},0}$  for $0\leq\tau\leq 1$,
where $\tilde{l}_{p_{\infty},0}$ is the reduced length with respect to $\tilde{g}_{\infty}(t)$. Hence
$\tilde{\mathcal {V}}_{p_{i},0}(\tau)\to \tilde{\mathcal {V}}_{p_{\infty},0}(\tau)$  for $0\leq\tau\leq 1$.  It follows from (\ref{3_2}) that
$\tilde{\mathcal {V}}_{p_{\infty},0}(1)=1$.
This implies $(M_{\infty},\tilde{h}_{\infty}(t))$ is the gradient solition with singular time at
$\tau=0$.
But then the only way for the curvature to stay bounded as
$\tau\to 0$ is if $(M_{\infty},\tilde{g}_{\infty}(t))$ is flat for all time. This contradicts (\ref{radius}) and thus proves the lemma.

\end{proof}

Now we give the proof of Theorem \ref{main_1}.	

\textbf{Proof of Theorem \ref{main_1}.}	

We argue by contradiction. If there exists  $\kappa$-solution of Ricci flow on 3-dimensional noncompact manifold with strictly positive sectional curvature and blow up at some finite time $T$ satisfying 	the condition
(\ref{curvature_rate})
for some point $p_0$.  Denote $h(\tau)=g(t)$ with $\tau(t)=T-t$. We divide the proof of Theorem \ref{main_1} into the following four steps. 

\textbf{Step 1}: By taking $s_i\to 0^-$, we show $l_{p_0,s_i}$ subconverges to a limit $l_{p_0,0}$.

 Take the constant path $\gamma$ in $p_0$, the condition (\ref{curvature_rate})
	implies
	\begin{align}
	l_{p_0,s_i}(p_0,\tau)\leq l(\gamma)\leq \frac{\int^T_{T-\tau}\sqrt{T-t}R(p_0,t)}{2\sqrt{\tau-s_i}}\leq C(\tau)
	\end{align}
	as $s_i\to 0$. 
	Combining this with (\ref{eq_l_1}), we obtain that there exists
	$l_{p_0,0}\in C^{0,1}(M\times (0,+\infty))$
	$$
	l_{p_0,s_i} \overset{}{\longrightarrow} l_{p_0,0},
	$$
	in $C^{0,\alpha}_{loc}(M\times (0,+\infty))$ and weakly in $W^{1,2}_{loc}(M\times (0,+\infty))$ with (\ref{eq_l_1}), (\ref{eq_l_2}),(\ref{estimate3}), (\ref{estimate1}),(\ref{estimate2}) hold and (\ref{eq_l_3}), (\ref{eq_l_4}) hold
	for distribution sense for $s=0$. Hence  $\mathcal {V}_{p,0}(\tau)\overset{def}{=}\int_{M} (4\pi \tau)^{-\frac{n}{2}}
	e^{-l_{p,0}(x,\tau)}
	dvol_{h(\tau)}(x)$ is monotone non-increasing under the backward Ricci flow $(M,h(\tau))$.
	
 By Theorem \ref{FIN},  for any fixed $\tau$ there exists $q^{s_i}_{\tau}\in M$  such that
	$l_{p,s_i}(q^{s_i}_{\tau},\tau)\leq \frac{n}{2}$. 
	Moreover, it follows from (\ref{estimate1})
	that
	$$
	cd_{h(\tau)}(p_0,q^{s_i}_{\tau})\leq 	(\tau-s_i)(l_{p_0,s_i}(p_0,\tau)+l_{p_0,s_i}(q^{s_i}_{\tau},\tau)+1)\leq C(\tau)
	$$
	as $s_i\to 0$. 
	Then for any fixed $\tau$, we have 	 $q^{s_i}_{\tau}\to q_{\tau}$ and
	\begin{align}
	l_{p_0,0}(q_{\tau},\tau)\leq \frac{n}{2}.
	\end{align}

\textbf{Step 2}:
Rescale the backward Ricci flow as   $h_i(\tau)=(\tau_i)^{-1}h(\tau_i\tau)$. 
Taking $\tau_i\to 0^-$ and $\tau_i\to +\infty$,
we show that $(M,h_i(\tau),q_{\tau_i})$ both subconverge to the  gradient shrinking solitons which is $\kappa-$noncollapsed on all scales.

Letting $s_i\to 0$ in (\ref{estimate1}) and (\ref{estimate2}), we obtain
\begin{align}
l_{p_0,0}(x,\tau)\geq -l_{p_0,0}(y,\tau)-1+c\frac{d^2_{h(\tau)}(x,y)}{\tau},\\
l_{p_0,0}^{\frac{1}{2}}(x,\tau) \leq l_{p_0,0}^{\frac{1}{2}} (y,\tau)+c\frac{d_{h(\tau)}(x,y)}{\sqrt{\tau}}.
\end{align}	
By the scaling property for $l_{p_0,0}$,
we obtain
\begin{align}
l^i_{p_0,0}(x,\tau)\geq -l^i_{p_0,0}(y,\tau)-1+c\frac{d^2_{h_i(\tau)}(x,y)}{\tau},\\
(l^i_{p_0,0})^{\frac{1}{2}}(x,\tau) \leq (l^i_{p_0,0})^{\frac{1}{2}} (y,\tau)+c\frac{d_{h_i(\tau)}(x,y)}{\sqrt{\tau}},
\end{align}	
where $l^i_{p_0,0}$ is the reduced length with respect to $h_i(\tau)$.
Since  $l_{p_0,0}(q_{\tau_i},\tau_i)\leq \frac{n}{2}$, 
$l^i_{p_0,0}(q_{\tau_i},1)\leq \frac{n}{2}$.  It follows from (\ref{estimate3}) and (\ref{eq_l_1})  that $l^i_{p_0,0}(q_{\tau_i},\tau)\leq C(\tau)$.	Then
\begin{align}
-C(\tau)-1+c\frac{d^2_{h_i(\tau)}(x,q_{\tau_i})}{\tau}\leq
l^i_{p_0,0}(x,\tau) \leq (C(\tau)^{\frac{1}{2}} +c\frac{d_{h_i(\tau)}(x,q_{\tau_i})}{\sqrt{\tau}})^2,\label{3_3}\\
|\nabla l^i_{p_0,0}|^2(x,\tau)+R_{h_i(\tau)}(x,\tau)\leq \frac{c}{\tau}(C(\tau)^{\frac{1}{2}} +c\frac{d_{h_i(\tau)}(x,q_{\tau_i})}{\sqrt{\tau}})^2.
\end{align}	
Hence $(M,h_i(\tau),q_{\tau_i})$ both subconverge to the $\kappa-$noncollapsed on all scales Ricci flows $(M^\pm_{\infty},h^\pm_{\infty}(\tau),q^\pm_{\infty})$ as  $\tau_i\to 0^-$ and $\tau_i\to +\infty$. 	Moreover, there exist
$l^\pm\in C^{0,1}(M^\pm_{\infty}\times (0,+\infty))$ such that
$
l^i_{p_0,0}\circ \phi_i \overset{}{\longrightarrow} l^{\pm}
$
in $C^{0,\alpha}_{loc}(M^\pm_{\infty}\times (0,+\infty))$ and weakly in $W^{1,2}_{loc}(M^\pm_{\infty}\times (0,+\infty))$ as  $\tau_i\to 0^-$ and $\tau_i\to +\infty$  with (\ref{eq_l_1}), (\ref{eq_l_2}),(\ref{estimate3}), (\ref{estimate1}),(\ref{estimate2}) hold and (\ref{eq_l_3}), (\ref{eq_l_4}) hold
for distribution sense for $s=0$.
 By (\ref{3_3}), scaling and monotone property of reduced volume,
we have
\begin{align}\label{3_4}
\mathcal {V}^{\infty}_+(\tau)=\lim\limits_{\tau_i\to+\infty}
\mathcal {V}^{\tau_i}_{p_0,0}(\tau)=\lim\limits_{\tau_i\to+\infty}\mathcal {V}_{p_0,0}(\tau_i\tau)\equiv c^+,
\end{align}
and
\begin{align}\label{3_5}
\mathcal {V}^{\infty}_-(\tau)=\lim\limits_{\tau_i\to 0^-}
\mathcal {V}^{\tau_i}_{p_0,0}(\tau)=\lim\limits_{\tau_i\to 0^-}\mathcal {V}_{p_0,0}(\tau_i\tau)\equiv c^-.
\end{align}
Then Theorem \ref{Perelman} implies
 $(M^\pm_{\infty},h^\pm_{\infty}(\tau),q^\pm_{\infty})$
are gradient shriking solitons and
$l^\pm$ are soliton functions satisfying
\begin{align}\label{3_7}
Rc_{h^\pm_{\infty}}+\text{Hess} l^\pm=\frac{h^\pm_{\infty}}{2\tau},
\end{align}
and
\begin{align}\label{3_6}
2\Delta l^\pm- |\nabla l^\pm|^2 + R_{h^\pm_{\infty}}+ \frac{l^\pm-3}{\tau}
= 0.
\end{align}

\textbf{Step 3}:
Next we show that the limits $(M^\pm_{\infty},h^\pm_{\infty}(\tau),q^\pm_{\infty})$ are non-flat.
If the limit shrinking soliton $(M^-_{\infty},h^-_{\infty}(\tau),q^-_{\infty})$ is 
flat, $(M^-_{\infty},h^-_{\infty}(\tau))$ is  isometric to Euclidean space.
This implies
 $\lim\limits_{\tau\to 0^-}\mathcal {V}_{p_0,0}(\tau)=1$ by (\ref{3_5}). Hence
 there exists $\tau_0>0$ such that $\mathcal{V}_{p_0,0}(\tau_0)\geq 1-\frac{\epsilon_0}{2}$. By
 Lebesgue theorem, we have $\mathcal{V}_{p_0,s_i}(\tau)\to \mathcal {V}_{p_0,0}(\tau)$ for any $\tau$ as $s_i\to 0$. Then there exists $N>0$ such that $\mathcal{V}_{p_0,s_i}(\tau_0)\geq 1-\epsilon_0$ when $i\geq N$. By Theorem \ref{regularity}, $r_{Rm}(p_0,s_i)\geq \epsilon_0 |\tau_0-s_i|$ for $i\geq N$.
 Hence $\limsup\limits_{s_i\to 0}|Rm|(p_0,s_i)\leq (\epsilon_0\tau_0)^{-2}$.
  This contradicts that $p_0$ is a blow-up point. Then we have proved that 
$(M^-_{\infty},h^-_{\infty}(\tau))$ is non-flat and $\lim\limits_{\tau\to 0^-}\mathcal {V}_{p_0,0}(\tau)<1$. Since $\mathcal {V}_{p_0,0}(\tau)$ is monotone 
non-increasing in $\tau$, we have $\lim\limits_{\tau\to +\infty}\mathcal {V}_{p_0,0}(\tau)<1$. Then by (\ref{3_4}), similar arguments show that $(M^+_{\infty},h^+_{\infty}(\tau))$ is non-flat.

\textbf{Step 4}:
Finally, we show that $(M,g(t))$ is shrinking soliton on $S^2\times \mathbb{R}$ which contradicts that  $(M, g(t))$ has the strictly positive sectional curvature.

 Note that only the non-flat $\kappa$-noncollapsed noncompact 3-dimensional shrinking soliton
 are  $S^2\times \mathbb{R}$, $RP^2\times \mathbb{R}$ and $(S^2\times \mathbb{R})/Z_2$(see \cite{P1}). Since $(M,g(t))$ has positive sectional curvature, it is diffeomorphic to $\mathbb{R}^3$, and $S^2\times \mathbb{R}$ is the only one that can
 arise as the limit of a sequence of Ricci flows that are diffeomorphic to $\mathbb{R}^3$. It follows that both  $(M^\pm_{\infty},h^\pm_{\infty}(\tau),q^\pm_{\infty})$ are $S^2\times \mathbb{R}$.
 
  Denoting by $x$ the coordinates on $S^2$ and by $y$ the coordinates on $\mathbb{R}$, the both potentials also can be splitted as $l^\pm(x,y,\tau)=l_{S^2}^\pm(x,\tau)+l_{\mathbb{R}}^\pm(y,\tau)$ with $l_{S^2}^\pm(x,\tau)$ and $l_{\mathbb{R}}^\pm(y,\tau)$ both satisfying the soliton equation
 (\ref{3_7})  on $S^2$ and $\mathbb{R}$ respectively.  
 By (\ref{3_7}) and  the compactness for  $S^2$,  we have both
 $l_{S^2}^\pm(x,1)$ are constants.
 It follows from  the soliton equation on $\mathbb{R}$  that we have
$l^\pm(x,y,1)=\frac{1}{4}y^2+c^\pm_1y+c^\pm_2$. 
Since $l^\pm(x,y,1)$ satisfies 
(\ref{3_6}), we conclude that
$c^\pm_2=(c^\pm_1)^2+1$ and hence
$l^\pm(x,y,1)=\frac{1}{4}|y+2c^\pm_1|^2+1$. Then $\mathcal {V}^{\infty}_-(\tau)=
\mathcal {V}^{\infty}_+(\tau)$ and
 $\lim\limits_{\tau\to +\infty}\mathcal {V}_{p_0,0}(\tau)=\lim\limits_{\tau\to 0^-}\mathcal {V}_{p_0,0}(\tau)<1$, 
and hence $\mathcal {V}_{p_0,0}(\tau)\equiv constant$.
It follows from Theorem \ref{Perelman} that
$(M,g(t))$ is shrinking soliton on $S^2\times \mathbb{R}$ which contradicts that  $(M, g(t))$ has the strictly positive sectional curvature.
	 $\Box$

\end{document}